\documentclass[12pt]{article}

\usepackage[T1]{fontenc}
\usepackage{lmodern}
\usepackage{textcomp}
\usepackage{amsfonts}
\usepackage[skins,breakable]{tcolorbox}
\usepackage{bm}
\usepackage{array}
\usepackage{booktabs}

\usepackage{amsmath}%
\usepackage{amssymb}%

\usepackage{todonotes}
\usepackage[in]{fullpage}%

\usepackage[amsmath,thmmarks]{ntheorem}%
\theoremseparator{.}%

\usepackage{titlesec}%
\titlelabel{\thetitle. }%
\usepackage{xcolor}%
\usepackage{mleftright}%
\usepackage{xspace}%
\usepackage{graphicx}
\usepackage{matlab-prettifier}

\usepackage{hyperref}%
\hypersetup{%
	unicode,
	breaklinks,%
	colorlinks=true,%
	urlcolor=[rgb]{0.25,0.0,0.0},%
	linkcolor=[rgb]{0.5,0.0,0.0},%
	citecolor=[rgb]{0,0.2,0.445},%
	filecolor=[rgb]{0,0,0.4},
	anchorcolor=[rgb]={0.0,0.1,0.2}%
}
\usepackage[ocgcolorlinks]{ocgx2}

\usepackage{cleveref}

\theoremseparator{.}%

\theoremstyle{plain}%
\newtheorem{theorem}{Theorem}[section]

\newtheorem{lemma}[theorem]{Lemma}

\newtheorem{corollary}[theorem]{Corollary}

\newtheorem{proposition}[theorem]{Proposition}

\theoremstyle{plain}%
\theoremheaderfont{\sf} \theorembodyfont{\upshape}%
\newtheorem*{remark:unnumbered}[theorem]{Remark}%
\newtheorem{remark}[theorem]{Remark}%

\newtheorem{example}[theorem]{Example}

\newcommand{\myqedsymbol}{\rule{2mm}{2mm}}

\theoremheaderfont{\em}%
\theorembodyfont{\upshape}%
\theoremstyle{nonumberplain}%
\theoremseparator{}%
\theoremsymbol{\myqedsymbol}%
\newtheorem{proof}{Proof:}%

\providecommand{\emphind}[1]{}%
\renewcommand{\emphind}[1]{\emph{#1}\index{#1}}

\definecolor{blue25emph}{rgb}{0, 0, 11}

\providecommand{\emphic}[2]{}
\renewcommand{\emphic}[2]{\textcolor{blue25emph}{%
		\textbf{\emph{#1}}}\index{#2}}

\providecommand{\emphi}[1]{}%
\renewcommand{\emphi}[1]{\emphic{#1}{#1}}

\definecolor{almostblack}{rgb}{0, 0, 0.3}

\providecommand{\emphw}[1]{}%
\renewcommand{\emphw}[1]{{\textcolor{almostblack}{\emph{#1}}}}%

\providecommand{\emphOnly}[1]{}%
\renewcommand{\emphOnly}[1]{\emph{\textcolor{blue25}{\textbf{#1}}}}

\newcommand{\HLink}[2]{\hyperref[#2]{#1~\ref*{#2}}}
\newcommand{\HLinkSuffix}[3]{\hyperref[#2]{#1\ref*{#2}{#3}}}

\providecommand{\deflab}[1]{}
\renewcommand{\deflab}[1]{\label{def:#1}}

\providecommand{\eqlab}[1]{}%
\renewcommand{\eqlab}[1]{\label{equation:#1}}

\newcommand{\remove}[1]{}%

\usepackage[inline]{enumitem}

\newlist{compactenumA}{enumerate}{5}%
\setlist[compactenumA]{topsep=0pt,itemsep=-1ex,partopsep=1ex,parsep=1ex,%
	label=(\Alph*)}%

\newlist{compactenuma}{enumerate}{5}%
\setlist[compactenuma]{topsep=0pt,itemsep=-1ex,partopsep=1ex,parsep=1ex,%
	label=(\alph*)}%

\newlist{compactenumI}{enumerate}{5}%
\setlist[compactenumI]{topsep=0pt,itemsep=-1ex,partopsep=1ex,parsep=1ex,%
	label=(\Roman*)}%

\newlist{compactenumi}{enumerate}{5}%
\setlist[compactenumi]{topsep=0pt,itemsep=-1ex,partopsep=1ex,parsep=1ex,%
	label=(\roman*)}%

\newlist{compactitem}{itemize}{5}%
\setlist[compactitem]{topsep=0pt,itemsep=-1ex,partopsep=1ex,parsep=1ex,%
	label=\ensuremath{\bullet}}%

\numberwithin{figure}{section}%
\numberwithin{table}{section}%
\numberwithin{equation}{section}%

\newcommand{\Sym}[1]{S_{#1}}

\usepackage{ragged2e}
\usepackage{marginnote}

  \newtcolorbox{cvbox}[1][]{
  enhanced,
    title=#1,
    breakable = true,
    overlay={%
        \ifcase\tcbsegmentstate
        \or%
        \else%
        \fi%
    }
    }

    \setlist[itemize]{
        itemsep=6pt,  % Space between individual items
        parsep=0pt,   % Space between paragraphs within an item
        topsep=12pt,  % Space before the first item of the list
        partopsep=0pt % Extra space added to topsep when a list starts a new paragraph
    }

\newcommand{\RaghavThanks}[1]{%
   \thanks{%
      Division of Science; %
      New York University; %
      Abu Dhabi, UAE;  %
      \href{mailto:r.tripathi@nyu.edu}{r.tripathi@nyu.edu}%
   #1%
   }%
}

\newcommand{\LudovickThanks}[1]{%
   \thanks{%
      Department of Mathematics and Statistics;
      Université Laval;
      \href{mailto:}{ludovick.bouthat.1@ulaval.ca} %
   #1%
   }%
}

\begin{document}

\title{Number of orbits of $k$-subsets of permutations}

\author{
Ludovick Bouthat
\LudovickThanks{}
\and 
Raghavendra Tripathi
\RaghavThanks{}
}

\date{}

\maketitle
\begin{abstract}
    Let $\Sym{n}$ denote the symmetric group of order $n$. Say that two subsets $x, y\subseteq \Sym{n}$ are \emph{equivalent} if there exist permutations $g_1, g_2\in \Sym{n}$ such that $g_1xg_2=y$, where multiplication is understood elementwise. Recently, [Tripathi, 2024] and [Kushwaha and Triathi, 2025] asked for the asymptotics of $T(n,k)$, the number of subsets of $\Sym{n}$ of size $k$ up to this equivalence. 
    It is easy to see that $T(n,0)=T(n, 1)=1$ and $T(n, 2)=p(n)-1$, where $p(n)$ is the number of integer partitions of $n$. In this work, we show that $T(n,k) = \Lambda_n(k)(1+o_n(1))$ for $3\leq k\leq n!-3$, where $\Lambda_n(k)=\frac{1}{n!^2}\binom{n!}{k}$. Furthermore, we prove that\vspace{-2pt}
    \[
\frac{1}{\Lambda_n(n!/2)}T\!\left(n,\left[\sqrt{\tfrac{n!}{4}}x+\tfrac{n!}{2}\right]\right) ~\xrightarrow{n\to\infty}~ \exp\!\left(-\tfrac{x^2}{2}\right)\,,
\]
uniformly over $\mathbb{R}$.
\end{abstract}

%%%%%%%%%%%%%%%%%%%%%%%%%%%%%%%%%%%%%%%%%%%%%%%%%%%%%%
%%%%%%%%%%%%%%%%%%%%%%%%%%%%%%%%%%%%%%%%%%%%%%%%%%%%%%
\section{Introduction}
\label{sec:introduction}

\subsection{Background and motivation}

% For $n\in \mathbb{N}$, let $\Sym{n}$ denote the permutation group on $[n]:=\{1,2, \ldots, n\}$. The symmetric group $\Sym{n}$ is a rich source of problems in combinatorics and probability. Many statistics of permutations naturally occur in surprisingly disparate areas of mathematics.  Naturally, the counting problems related to the statistics of symmetric groups are too vast to survey here. We thus refer the reader to~\cite{Ford2022Cycle,arratia1997random, flajolet2009analytic,van2001course} and the references therein. One particularly interesting and important class of such problems is counting the orbits of the action of the symmetric group $\Sym{n}$. 

For $n \in \mathbb{N}$, let $\Sym{n}$ denote the symmetric group on $[n] := \{1, 2, \ldots, n\}$. The symmetric group $\Sym{n}$ is a rich source of problems in combinatorics and probability, with many permutation statistics arising naturally in surprisingly disparate areas of mathematics. Counting problems related to such statistics are too numerous to survey here; we thus refer the reader to~\cite{Ford2022Cycle, arratia1997random, flajolet2009analytic, van2001course} and the references therein for comprehensive treatments. Among these problems, a particularly interesting and important class concerns the enumeration of orbits under various natural actions of $\Sym{n}$ on some finite set.

% We study the number of orbits of the action of $\Sym{n}\times \Sym{n}$ on the $k$-subsets of $\Sym{n}$.
% %For instance, the symmetric group $\Sym{n}$ acts on itself via conjugation. The orbits of this action, \emph{conjugacy classes}, are in one-to-one correspondence with the integer partitions of $n$. The motivation for our work comes from the questions raised in~\cite{tripathi2025some,kushwaha2025note} recently. 
% For $1\leq k\leq n!$, let $\Omega(n, k)$ denote the collection of subsets of $\Sym{n}$ of size $k$. Let $G=\Sym{n}\times \Sym{n}$. Let $(g_1, g_2)\in G$ and let $x=\{x_1, \ldots, x_k\}\in \Omega(n, k)$. We define the action
% \[(g_1, g_2)\cdot x = \{g_1xg_2, \ldots, g_1xg_2\}\;.\]
% How many orbits are there? Or, equivalently, how many elements $x\in \Omega(n, k)$ are fixed on average by a random $(g_1, g_2)\in G$? When $k=1$, it is easily seen that this action is transitive, that is, there is only one orbit. However, this action is not transitive for $k\geq 2$. 
Our goal is to study the number of orbits of the action of $\Sym{n} \times \Sym{n}$ on the collection of $k$-subsets of $\Sym{n}$. For $0 \leq k \leq n!$, let $\Omega(n, k)$ denote the family of all subsets of $\Sym{n}$ of size $k$. Let $(g_1, g_2) \in \Sym{n} \times \Sym{n}$ and $x = \{x_1, \ldots, x_k\} \in \Omega(n, k)$. Define the action
\[
(g_1, g_2) \cdot x := \{g_1 x_1 g_2, \ldots, g_1 x_k g_2\}\;.
\]
When $k = 1$, this action is easily seen to be transitive; that is, there is only one orbit. However, the action is no longer transitive when $k \geq 2$.
How many orbits does this action have for $k\geq 2$? Equivalently, what is the expected number of fixed points of a uniformly random $(g_1, g_2) \in \Sym{n} \times \Sym{n}$ acting on $\Omega(n, k)$? 

For $0 \leq k \leq n!$, let $T(n, k)$ denote the number of orbits of the action of $\Sym{n} \times \Sym{n}$ on $\Omega(n, k)$. The problem of obtaining good asymptotics for $T(n, k)$ was raised in~\cite{kushwaha2025note}. It was already observed in~\cite{tripathi2025some} that $T(n, 2) = p(n) - 1$ and $T(n, k) = T(n, n! - k)$ for all $k$, where $p(n)$ is the number of partitions of the integer $n$. For small values of $n$, the exact values of $T(n, k)$ are known. In particular, for $n = 3$ and $n = 4$, these values are recorded in \Cref{tab:table for S3}. %(see also~\cite[Section~4]{tripathi2025some})  
 Note that for $n = 4$, we list values only up to $k = 12 = n!/2$ since $T(n, k)$ is symmetric about $n!/2$. For $n \leq 6$, the values of $T(n, k)$ appear in~\cite{OEIS_setEquivalent}. Computing the number of orbits even for moderately large $n$ appears to be intractable. 
\begin{table}[h!]
    \centering
    \begin{tabular}{|c||c|c|c|c|c|c|c|c|c|c|c|c|c|}
    \hline 
        $k$ & 0 & 1  & 2  & 3  & 4  & 5  & 6 & 7 & 8 & 9 & 10 & 11 & 12 \\
    \hline &&&&&&&&&&&&&\\[-12.5pt]
    \hline 
       $T(3, k)$ & 1 & 1 & 2 & 2 & 2 & 1 & 1 &&&&&&\\
    \hline
    $T(4, k)$ & 1 & 1 & 4 & 10 & 41 & 103 & 309 & 691 & 1458 & 2448 & 3703 & 4587 & 5050 \\
    \hline
    \end{tabular}
    \caption{Number of orbits of $k$-subsets of $\Sym{3}$ and $\Sym{4}$}
    \label{tab:table for S3}
\end{table}
\vspace{-1.5pt}

Since $T(n, k)$ is the expected number of fixed points of a uniformly random $(g_1, g_2) \in \Sym{n} \times \Sym{n}$ acting on $\Omega(n, k)$, a trivial lower bound is
\[
T(n, k) \geq \frac{1}{n!^2} \binom{n!}{k} =: \Lambda_n(k),
\]
which corresponds to the scenario where all $n!^2 - 1$ non-identity pairs $(g_1, g_2) \in \Sym{n} \times \Sym{n}$ fix no element of $\Omega(n, k)$, while the identity pair $(id, id)$ fixes all $\binom{n!}{k}$ elements. This bound is far from sharp when $k = 0$, $1$, or $2$, as the non-identity pairs of $\Sym{n} \times \Sym{n}$ fix \emph{too many} elements of $\Omega(n, k)$ when $k$ is small. However, it would be close to the true value if most orbits had a size of the order of $n!^2$. One of our main results confirms that this is indeed the case for $3\leq k \leq n!-3$.

\vspace{-1.5pt}
\subsection{Main results}

\begin{theorem}[Asymptotic of $T(n, k)$]
\label{thm:Asymptotic}
Let $T(n, k)$ be as above, and let $3\leq k \leq n!-3$. Then, as $n\to \infty$, we have 
\[ 
T(n, k) =  \frac{1}{n!^2}\binom{n!}{k}(1+o_n(1))\,.
\]
Moreover, the $o_n(1)$ term can be taken to be $O(n^{-2})$. In that case, both the domain $[3, n! - 3]$ and the quadratic rate of convergence cannot be made larger.
\end{theorem}

%\begin{remark}
%    \Cref{thm:Asymptotic} and its proof suggest the possibility that, for each fixed $n$, the ratio $T(n, k)/\Lambda_n(k)$ may admit an expansion in powers of $1/n$. That is,
%    \[
%    T(n, k) = \Lambda_n(k) \bigg(1 + \sum_{j\geq 1} A^{(n,k)}_j \frac{1}{n^j} \bigg).
%    \]
%    Such asymptotic expansions frequently appear in the context of partition functions of Gibbs measures in statistical physics; see, for instance,~\cite{Geloun2022}. The coefficients in these expansions often have clean combinatorial interpretations. Establishing such an expansion for $T(n, k)$ would be an interesting direction for further research.
%\end{remark}

%Using Paley-Zygmund inequality and Theorem~\ref{thm:Asymptotic}, we conclude the following analog of the Besicovitch problem. \textcolor{red}{I'm not sure I understand why this result holds... Also, shouldn't the inequality be reversed?}\textcolor{blue}{You are right. We may have to drop this.}
%\begin{corollary}
%    \label{cor:No fixed point}
%   For $g=(g_1, g_2)\in \Sym{n}\times \Sym{n}$, let \[\mathcal{F}_{k}(g) = \{x\in \Omega(n, k): g\cdot x=x \}\]
%be the set of fixed points of $g$. Let $g$ be chosen uniformly at random from $\Sym{n}\times \Sym{n}$. If $3\leq k \leq n!-3$, then 
%\[\mathbb{P}(|\mathcal{F}_{k}(g)|=0) \geq  1- \frac{1}{n!^2}(1+o_n(1)) \;. \]
%\end{corollary}

\Cref{thm:Asymptotic} shows that $T(n, n!/2)\sim \Lambda_n := \frac{1}{n!^2}\binom{n!}{n!/2}$. The following theorem gives a general scaling limit for $T(n, k)$ near $k\approx n!/2$.
\begin{corollary}[Scaling limit]
   \label{thm - bulk asymptotic}
   For $x\in \mathbb{R}$, define $k_n(x) := \big\lfloor\frac{n!}{2}+ x\sqrt{\frac{n!}{4}} \big\rfloor$ and
    \[
    B_n(x) := \frac{T(n, k_n(x))}{\Lambda_n}\;,
    \]
    where $T(n, k):= 0$ if $k\notin [n!]$. Then, $B_n(x)\to \exp\!\left(-\frac{x^2}{2}\right)$, uniformly over $\mathbb{R}$, as $n\to \infty$.
\end{corollary}
\begin{proof}
The classical De Moivre--Laplace theorem states that
\begin{align*}
    \frac{\sqrt{\pi n!/2}}{2^{n!}} \binom{n!}{k_n(x)} =(1+o_n(1)) \exp\!\Big(\!-\frac{(k_n(x)-n!/2)^2}{n!/2}\Big).
\end{align*}
The result thus follows from \Cref{thm:Asymptotic} since $B_n(x) = \frac{T(n, k_n(x))}{T(n, k_n(0))} = (1+o_n(1))\binom{n!}{k_n(x)}\!\big/\!\binom{n!}{k_n(0)}$.\,
\end{proof}

\begin{figure}
\begin{tikzpicture}
  \node[anchor=center, inner sep=0](image) at (0,0){\includegraphics[width=\linewidth]{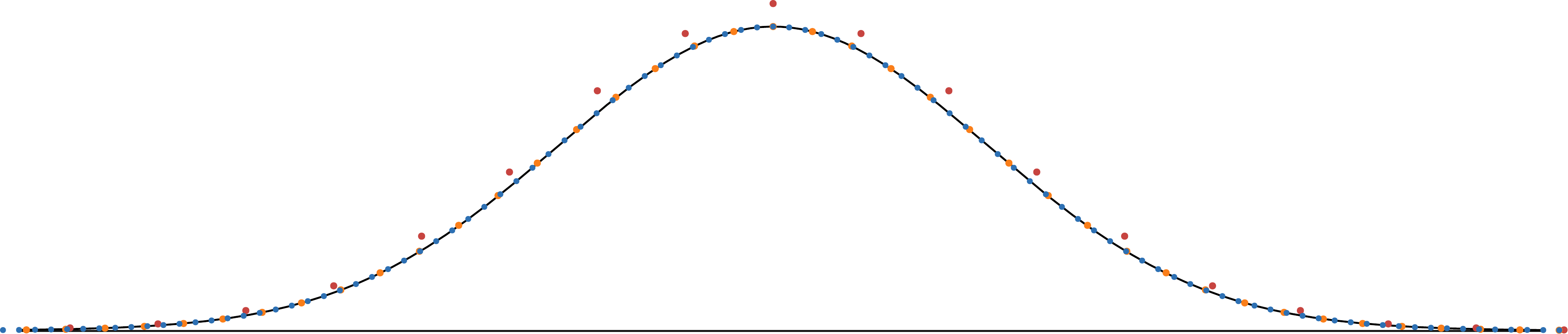}
  };
  % Add legend entries (adjust coordinates as needed)
 \begin{scope}[xshift=6cm, yshift=2cm]
      % Red dot
      \filldraw[red] (0,0) circle (1pt);
      \node[right] at (0.3,0) {$n=4$};

      % Blue dot
      \filldraw[orange] (0,-0.5) circle (1pt);
      \node[right] at (0.3,-0.5) {$n=5$};

      % Green dot
      \filldraw[blue] (0,-1) circle (1pt);
      \node[right] at (0.3,-1) {$n=6$};
    \end{scope}
\end{tikzpicture}
\caption{Convergence of $B_n(x)\to \exp(-x^2/2)$.}
\end{figure}

As an immediate consequence of \Cref{thm - bulk asymptotic}, we obtain the following remarkable result. 
\begin{corollary}[Asymptotic log-concavity of $T(n, k_n(x))$ in the bulk]\label{cor - log-concavity}
For any real numbers $x$ and $y$, we have  
\[ \lim_{n\to \infty} \frac{T(n, k_n(x))^2}{T(n, k_n(x+y)) T(n, k_n(x-y))} = e^{y^2} \geq 1\;.\]
\end{corollary}
\begin{proof}
A direct computation reveals that 
\begin{align*}
    \lim_{n\to \infty} \frac{T(n, k_n(x))^2}{T(n, k_n(x+y)) T(n, k_n(x-y))} &= \lim_{n\to \infty} \frac{\frac{T(n, k_n(x))^2}{\Lambda_n^2}}{\frac{T(n, k_n(x+y))}{\Lambda_n} \frac{T(n, k_n(x-y))}{\Lambda_n}}= \frac{e^{-x^2}}{e^{-x^2-y^2}} = e^{y^2}.
\end{align*}

%which proves Corollary~\ref{cor - log-concavity}.
\end{proof}

%For any $k$, we have a trivial lower bound $\frac{1}{n!^2}\binom{n!}{k}$ on the number of orbits. %This suggests that for $k\geq 3$, the probability that a random pair of permutation $(g_1, g_2)$ does not fix any $k$-set is $1-o_n(1)$. This is reminiscent of the famous Besicovitch problem\cite{Besicovitch1935} about the probability that a random permutation does not fix any $k$-subset $[n]$. This problem has deep connections to number theory and we refer the readers to~\cite{arratia1997random, Ford2022Cycle, ford2006integers,eberhard2016permutations} for more details.

%%%%%%%%%%%%%%%%%%%%%%%%%%
\subsection{Outline of the paper}

To estimate $T(n,k)$, we reduce to counting fixed $k$-sets under the action of $(g_1, g_2)\in \Sym{n} \times \Sym{n}$ via Burnside’s lemma.
This leads us in \Cref{sec:Proof of General action} to finding the number of $k$-sets fixed under some group action. In particular, \Cref{thm:GeneralAction} gives a formula for this. This result is of independent interest and it generalizes \cite[Theorem 4.1]{diaconis2008fixed}. 

Using this, we obtain a formula for $T(n,k)$ that depends on the number of permutations having cycle type $\nu^m$, which is defined as the cycle type of the permutation $g^m$, where $g$ is any permutation of cycle type $\nu$. In \Cref{lem - ratio}, we obtain a bound on this in terms of the number of permutations with cycle type $\nu$. This result may also be of independent interest.  

%The rest of the paper is essentially dedicated to the proof of \Cref{thm:Asymptotic}. 
In \Cref{sec: Aux}, we establish some preliminary lemmas which are crucial to the proof of our main result. 
%In particular, we non-trivially relate the number of permutations having cycle type $\nu^m$ to the number of those having cycle type $\nu$ in \Cref{lem - ratio}, which may also be a problem of independant interest. Having established these auxiliary results, 
We then finally prove \Cref{thm:Asymptotic} in \Cref{sec - sharp_upperbound}. 
%The proof is technical and relies on delicate estimations of several subcases of \Cref{cor - two-side}. 

% In Section~\ref{sec:Proof of General action}, we first prove Theorem~\ref{thm:GeneralAction}. The rest of the paper is essentially about the proof of Theorem~\ref{thm:Asymptotic}, which relies on several other propositions that in turn assume Lemma~\ref{lem-partition} and Proposition~\ref{prop - ratio}. The proof of Theorem~\ref{thm:Asymptotic} appears in Section~\ref{sec - sharp_upperbound}. Finally, in Section~\ref{sec: Aux} we prove Lemma~\ref{lem-partition} and Proposition~\ref{prop - ratio}.  

%%%%%%%%%%%%

\section{Counting Fixed \texorpdfstring{$k$}{k}-Subsets under Group Actions}

\subsection{A general result}

Recall our main objective to count the number $T(n,k)$ of orbits of the action of $\Sym{n}\times \Sym{n}$ on the collection of $k$-subsets of $\Sym{n}$. To this end, a natural approach is to use the classical \emph{Burnside's lemma}, also known as the \emph{orbit-counting theorem}. 
Now, if $\mathcal{G}$ is a group acting on a set $X$, then Burnside's lemma states that the number of orbits satisfies
\begin{equation}\label{eqn:Burnside}
    \# \text{Orbits} = \frac{1}{|\mathcal{G}|} \sum_{g \in \mathcal{G}} |F_g|,
\end{equation}
where $F_g := \{x \in X : g \cdot x = x\}$ is the set of fixed points of $g$. 
Hence, the problem of counting fixed $k$-subsets under group actions essentially comes down to counting the number of fixed points in $X$ under the action $g\in\mathcal{G}$, for all possible $g\in\mathcal{G}$. 
To this end, we first consider a more general setting. Hence, let $\mathcal{G}$ be a finite group acting on a finite set $X$ via automorphisms; that is, the map
\[
x \mapsto g \cdot x
\]
is an automorphism of $X$ for each $g \in \mathcal{G}$. Such an action naturally extends to an action of $\mathcal{G}$ on $\binom{X}{k}$, the family of all $k$-element subsets of $X$. 
Moreover, for $g \in \mathcal{G}$ and $k \in \mathbb{N}$, let
\[
\binom{X}{k}^{\!\raisebox{-2pt}{$\scriptstyle g$}} := \left\{x \in \binom{X}{k} : g \cdot x = x \right\}
\]
denote the collection of fixed points of $g$ in $\binom{X}{k}$. The following theorem provides an explicit formula for the cardinality of this set in terms of the sets
\begin{equation}\label{eq - def_Tg(p)}
    T_g(j) := \left\{x \in X : g^j \cdot x = x \,\text{ and }\, g^i \cdot x \neq x \text{ for all } 1 \leq i < j \right\},
\end{equation}
which consist of all elements of $X$ that are fixed by $g^j$ but not by any $g^i$ for $1 \leq i < j$.

\begin{theorem}[Number of $k$-subsets fixed by $g$]
\label{thm:GeneralAction}
    Let $X$ be a finite set and let $\mathcal{G}$ be a group acting on $X$ by automorphisms. Then
    \[
     \left|\binom{X}{k}^{\!\raisebox{-2pt}{$\scriptstyle g$}}\right| = \sum_{\lambda \vdash k} \prod_{j=1}^{k} \binom{|T_{g}(j)|/j}{m_j(\lambda)},
    \]
    where the sum is over all integer partitions $\lambda$ of $k$ and $m_j(\lambda)$ denotes the number of parts of size $j$ in $\lambda$. 
    Here, we use the convention that $\binom{\alpha}{\beta}=0$ if $\alpha<\beta$. 
    % \[|T_{g}(j)| = \sum_{m\mid j}|\widetilde{T}_{g}(m)|\mu\big(\frac{j}{m}\big)\;,\]
    % where $\mu$ is the Mobius function introduced in~\eqref{eq:def_Mobius}.
\end{theorem}

\iffalse
\begin{remark}
\label{remark:Mobius}
Instead of $T_g(j)$, it is often easier to estimate the size of the set
\[
\widetilde{T}_g(j) := \{x \in X : g^j \cdot x = x\}.
\]
Note that we trivially have $T_g(j) \subseteq \widetilde{T}_g(j)$. Moreover, the M\"obius inversion formula~\cite{bender1975applications} allows us to express $|T_g(j)|$ in terms of the cardinalities of the sets $\widetilde{T}_g(m)$:
\begin{equation}
\label{eq - T_general}
|T_g(j)| = \sum_{m \mid j} \mu(j/m) \, |\widetilde{T}_g(m)|,
\end{equation}
where $\mu$ denotes the classical M\"obius function
\begin{equation}
\label{eq:def_Mobius}
\mu(x) = 
\begin{cases}
1 & \text{if } x = 1, \\
(-1)^p & \text{if } x \text{ is a product of } p \text{ distinct primes}, \\
0 & \text{if } x \text{ is divisible by the square of a prime}.
\end{cases}
\end{equation}
% These observations shall prove to be useful in \Cref{ } to provide a proof of \Cref{cor - two-side}.
\end{remark}
\fi

We shall soon prove \Cref{thm:GeneralAction} in \Cref{sec:Proof of General action} and use it in \Cref{sec - two-sided} to obtain an explicit formula for $T(n,k)$ (\Cref{cor - two-side}). However, let us first look at an interesting special case, in which we retrieve a result from Diaconis, Fulman, and Guralnick.
\medskip

\begin{example}

Let $\mathcal{G} = \Sym{n}$ be the symmetric group, and let $X = [n] = \{1, \ldots, n\}$. The group $\mathcal{G}$ naturally acts on $[n]$ via the map $\sigma: i \mapsto \sigma(i)$. This action extends to the family of $k$-subsets of $[n]$, denoted $\binom{[n]}{k}$, by
\[
\sigma \cdot \{x_1, \ldots, x_k\} = \{\sigma(x_1), \ldots, \sigma(x_k)\}.
\]
Let $T_\sigma(j)$ be defined in the obvious way as in \eqref{eq - def_Tg(p)}. Then it is easily seen that $T_\sigma(j)$ is the set comprised of all the numbers in $[n]$ which belong to a $j$-cycles of $\sigma$. In particular,
\[
\frac{|T_\sigma(j)|}{j} = \, \text{number of $j$-cycles in } \sigma  \,=:c_j(\sigma).
\]
Let $\binom{[n]}{k}^{\smash{\sigma}}$ denote the number of $k$-subsets of $[n]$ fixed by $\sigma \in \Sym{n}$.\Cref{thm:GeneralAction} yields
\begin{equation}
\label{eqn:DFG_Formula}
\binom{[n]}{k}^{\!\raisebox{-2pt}{$\scriptstyle \sigma$}} = \sum_{\lambda \vdash k} \prod_{j = 1}^{k} \binom{c_j(\sigma)}{m_j(\lambda)},
\end{equation}
which is precisely \cite[Theorem 4.1]{diaconis2008fixed}.

A classical result of Goncharov\cite{Goncharev} states that for a uniformly random permutation $\sigma \in \Sym{n}$, and for fixed $k$, the joint distribution of the cycle counts $(c_j(\sigma))_{1 \leq j \leq k}$ converges, as $n \to \infty$, to a vector $(X_j)_{1 \leq j \leq k}$ of independent Poisson random variables with means $\mathbb{E}[X_j] = 1/j$. Using this fact together with~\eqref{eqn:DFG_Formula}, the authors further obtained the limiting distribution of $F_{\sigma}(k)$ as $n \to \infty$ for all fixed $k$; see~\cite[Theorem 4.1]{diaconis2008fixed}.
\end{example}

\subsection{Proof of \texorpdfstring{\Cref{thm:GeneralAction}}{Theorem 2.1}}
\label{sec:Proof of General action}

We now prove \Cref{thm:GeneralAction}. 
Let $g\in\mathcal{G}$ be given. Observe that $\widetilde{x} = \{x_1,\dots,x_k\}\in \binom{X}{k}^{\!g}$ if and only if there is a permutation $\pi\in \Sym{k}$ such that
\[g \cdot x_{i} = x_{\pi(i)}, \quad 1\leq i\leq k\;.\]
Suppose $\pi\in \Sym{k}$ is such a permutation and let $c=(i_1, \ldots, i_{t})$ be a cycle of length $t$ in $\pi$. We then have $g\cdot x_{i} = x_{\pi(i)}$ for $i=1,2,\dots,k$ if and only if
\[g \cdot x_{i_s} = x_{i_{s+1}}, \quad s\in [t]\,,\]
with the convention that $i_{t+1}=i_1$, for all cycles in $\pi$. Note that the permutation $\pi$ depends on how we enumerate the elements in $x$. However, the cycle type of permutation $\pi$ is independent of the choice of enumeration. Recall that the set of cycle structures of permutations $\pi\in \Sym{k}$ is in one-to-one correspondence with the set of integer partitions $\lambda\vdash k$, where the parts of the partition $\lambda$ correspond to the length of the cycles of $\pi$. For an integer partition $\lambda \vdash k$ and $g\in \mathcal{G}$, let us define 
\[
\binom{X}{k}^{\!\raisebox{-2pt}{$\scriptstyle g$}}_{\!\raisebox{2pt}{$\scriptstyle \lambda$}} := \left\{\widetilde{x}\in \binom{X}{k}^{\!\raisebox{-2pt}{$\scriptstyle g$}} : [g, \widetilde{x}] = \lambda\right\}\;,
\]
where the notation $[g, \widetilde{x}] = \lambda$ means that for some (and hence any) enumeration of the set $\widetilde{x}$, the induced permutation on $[k]$ by the action of $g$ on $\widetilde{x}$ has a cycle structure given by $\lambda$. With this notation, we readily obtain
\begin{equation}
\label{eqn:DecompositionInLambda}
    \left|\binom{X}{k}^{\!\raisebox{-2pt}{$\scriptstyle g$}} \right| = \sum_{\lambda\vdash k} \left|\binom{X}{k}^{\!\raisebox{-2pt}{$\scriptstyle g$}}_{\!\raisebox{2pt}{$\scriptstyle \lambda$}} \right|\;.
\end{equation}
On the set $T_{g}(p) \subseteq X$, defined in \eqref{eq - def_Tg(p)}, we consider the equivalence relation $x_1 \sim x_2$ if $x_1 = g^s \cdot x_2$ for some $s\in \mathbb{Z}$. Let
\[
S_{g}(p) = T_{g}(p)/\sim
\]
denote the set of equivalence classes in $T_{g}(p)$ under this relation. 
By definition, every $x \in T_g(p)$ satisfies $g^p \cdot x = x$ and $g^i \cdot x \neq x$ for all $1 \leq i < p$. This implies that the orbit of $x$ under the action of $g$ consists of exactly $p$ distinct elements, and that each equivalence class contains precisely $p$ elements. Therefore, we have the identity $|S_g(p)| = |T_g(p)|/p.$ 
We now argue that
\begin{equation}
    \label{eqn:FixedwithLambda2}
    \left|\binom{X}{k}^{\!\raisebox{-2pt}{$\scriptstyle g$}}_{\!\raisebox{2pt}{$\scriptstyle \lambda$}} \right| = \prod_{j=1}^k \binom{|S_{g}(j)|}{m_j(\lambda)} = \prod_{j=1}^k \binom{|T_{g}(j)|/j}{m_j(\lambda)}\;.
\end{equation}
To this end, observe that $\widetilde{x} \in \binom{X}{k}^{\hspace{-.5pt}\raisebox{-1pt}{$\scriptstyle g$}}_{\!\raisebox{2pt}{$\scriptstyle \lambda$}}$ if and only if $\widetilde{x}$ is a disjoint union of $g$-orbits (or cycles) of prescribed lengths, that is,
\[
\widetilde{x} = \bigsqcup_{j=1}^{k} \bigsqcup_{m=1}^{m_{j}(\lambda)} C_g(j, m),
\]
where each $C_g(j, m)$ is a $g$-orbit of length $j$ of the form $(x_1, \ldots, x_j)$ satisfying $g \cdot x_t = x_{t+1}$, with indices taken modulo $j$ (i.e., $x_{j+1} = x_1$). 
Each such cycle essentially corresponds to an equivalence class in $S_g(j)$. Moreover, since the cycles are disjoint, choosing $m_j(\lambda)$ of them corresponds to selecting $m_j(\lambda)$ distinct elements from $S_g(j)$, which can be done in $\binom{|S_g(j)|}{m_j(\lambda)}$ ways. Hence, taking the product over $j = 1,2,\dots,k$ gives the formula in \eqref{eqn:FixedwithLambda2}, and completes the proof.

\subsection{Two sided multiplication action}\label{sec - two-sided}
Let us now return to the problem of main interest to this paper, namely computing $T(n,k)$, which corresponds to the case $\mathcal{G}=\Sym{n}\times \Sym{n}$, $X=\Sym{n}$ in \Cref{thm:GeneralAction}. 
For $g=(g_1, g_2)\in \Sym{n}\times\Sym{n}$, recall that we have 
\begin{equation}\label{eq - T}
    \begin{gathered}
    T_g(j)=:T_{g_1, g_2}(j) = \{\sigma\in \Sym{n}: g_1^j\sigma g_2^j=\sigma \,\text{ and }\, g_1^{i}\sigma g_2^{i}\neq \sigma \,\text{ for all }\, 1\leq i< j\}\,,\\
    \widetilde{T}_g(j)=:\widetilde{T}_{g_1, g_2}(j) = \{\sigma\in \Sym{n}: g_1^j\sigma g_2^j=\sigma\}\,.
\end{gathered}
\end{equation}
We claim that $|T_{g_1, g_2}(j)|$ (resp. $|\widetilde{T}_{g_1, g_2}(j)|$) is completely determined by the cycle type of $g_1$ and $g_2$. Indeed, let $(g_1, g_2)$ and $(g_1', g_2')$ be two elements of $\Sym{n} \times \Sym{n}$ with the same respective cycle types $(\nu_1, \nu_2)$. Then it follows that $g_1$ and $g_1'$ are conjugate, as are $g_2$ and $g_2'$. Hence, choose $h_1, h_2 \in \Sym{n}$ such that
\[
h_1 g_1 h_1^{-1} = g_1' \quad\text{ and }\quad  h_2 g_2 h_2^{-1} = g_2',
\]
and observe that for any fixed $j$, the map $\sigma \mapsto h_1^j \sigma h_2^{-j}$ defines a bijection between $T_{g_1, g_2}(j)$ and $T_{g_1', g_2'}(j)$ (resp. $\widetilde{T}_{g_1, g_2}(j)$ and $\widetilde{T}_{g_1', g_2'}(j)$). Hence, it follows that $\lvert T_{g_1, g_2}(j) \rvert$ (resp. $\lvert \widetilde{T}_{g_1, g_2}(j) \rvert$) depends only on the pair $(\nu_1, \nu_2)$, and not on the specific permutations with these cycle types. By abuse of notation, we shall thus denote this quantity by $\lvert T_{\nu_1, \nu_2}(j) \rvert$ (resp. $\lvert \widetilde{T}_{g_1, g_2}(j) \rvert$).

For a permutation $g$ with cycle type $\nu$, denote by $\nu^m$ the cycle type associated to the permutation $g^m$. Moreover, define in the usual way the quantity $z_{\nu} = \prod j^{m_j(\nu)} m_j(\nu)!$. It is a standard fact that the number of permutations $\sigma \in \Sym{n}$ with cycle type $\nu$ is given by $n! / z_{\nu}$. 
Combining this fact with \Cref{thm:GeneralAction}, Burnside's lemma~\eqref{eqn:Burnside} we obtain the following orbit-counting formula for the two-sided action.

% \begin{align*}
%     T(n,3) &=\sum_{\nu\vdash n} \frac{1}{z_{\nu}} \left( \frac{z_{\nu}^2-6z_{\nu}}{6} +\frac{z_{\nu^2}}{2} \right)   +  \frac{1}{3}\sum_{(\nu_1, \nu_2)\vdash n} \frac{z_{\nu_1^3} \delta_{\nu_1^3,\nu_2^3}}{z_{\nu_1} z_{\nu_2}}  \\
%     %
%     &= \frac{1}{2} \sum_{\nu\vdash n} \left( \frac{z_{\nu}}{3} +\frac{z_{\nu^2}}{z_{\nu}} \right)   +  \frac{1}{3}\sum_{(\nu_1, \nu_2)\vdash n} \frac{z_{\nu_1^3} \delta_{\nu_1^3,\nu_2^3}}{z_{\nu_1} z_{\nu_2}} - p(n)  \\
%     %
%     &=  \sum_{\nu\vdash n} \left( \frac{z_{\nu}}{6} +\frac{z_{\nu^2}}{2z_{\nu}} + \frac{z_{\nu^3}}{3z_{\nu}^2} \right) +  \frac{1}{3}\sum_{\nu_1\neq \nu_2\vdash n} \frac{z_{\nu_1^3} \delta_{\nu_1^3,\nu_2^3}}{z_{\nu_1} z_{\nu_2}} - p(n)
% \end{align*}

\begin{corollary}\label{cor - two-side}
    Let $n\geq 0$ and $0\leq k \leq n!$ be integers. Then
    \begin{equation*}
    % \label{eqn: Orbit formula for two-sided action}
        T(n, k) = \sum_{(\nu_1, \nu_2)\vdash n} \frac{1}{z_{\nu_1} z_{\nu_2}} \sum_{\lambda \vdash k} \prod_{j=1}^{k} \binom{|T_{\nu_1, \nu_2}(j)|/j}{m_j(\lambda)}\,,
    \end{equation*}
    where $|T_{\nu_1, \nu_2}(j)| = \sum_{m|j} \mu(j/m) z_{\nu_1^m} \delta_{\nu_1^m,\nu_2^m}$ and $\delta_{i,j}$ is the Kronecker delta symbol.
\end{corollary}
\begin{proof}
Using \Cref{thm:GeneralAction} and Burnside's lemma~\eqref{eqn:Burnside}, we obtain
\begin{equation*}
    T(n, k) = \sum_{(\nu_1, \nu_2)\vdash n} \frac{1}{z_{\nu_1} z_{\nu_2}} \sum_{\lambda \vdash k} \prod_{j=1}^{k} \binom{|T_{\nu_1, \nu_2}(j)|/j}{m_j(\lambda)}\;.
\end{equation*}
Finally observe that the classical M\"obius inversion~\cite{bender1975applications} yields
\[
|T_{\nu_1, \nu_2}(j)| = \sum_{m|j} \mu(j/m) |\widetilde{T}_{\nu_1,\nu_2}(m)|.
\]
Finally, note that $\widetilde{T}_{\nu_1, \nu_2}(m)\neq \emptyset$ if and only if $\nu_1^m=\nu_2^m$. Now suppose $\nu_1=\nu_2=\nu$. We know that for any $g_1, g_2\in \Sym{n}$ with cycle types $\nu$, we have 
\[|\widetilde{T}_{\nu_1, \nu_2}(m)| = \{\sigma\in \Sym{n}: g_1^m\sigma g_2^m=\sigma\}\;.\]
Taking $g_1=g_2$ immediately shows that $|\widetilde{T}_{\nu_1, \nu_2}(m)|=z_{\nu^m}$. This completes the proof.
%Hence, let $g_1,g_2\in \Sym{n}$ be two permutations of cycle types $\nu_1$ and $\nu_2$, respectively, and suppose that $g_1^j$ and $g_2^j$ have \emph{distinct} cycle type (i.e., $\nu_1^j \neq \nu_2^j$). Then it follows that $T_{\nu_1, \nu_2}(j) = \emptyset$, since otherwise we would have $g_1^j \sigma g_2^j = \sigma$, which would imply that $g_1^j$ and $g_2^j$ are similar and thus have the same cycle type, a contradiction. Hence, $|T_{\nu_1, \nu_2}(j)| =0$ if $\nu_1^j\neq \nu_2^j$.  
%
%Otherwise, if $\nu_1^j=\nu_2^j$, then $g_1^j$ is similar to $g_2^j$ (and thus to $g_2^{-j}$, which has the same cycle type as $g_2^j$), and $\widetilde{T}_{\nu_1,\nu_2}(j)$ consists precisely of all the permutations $\sigma\in\Sym{n}$ which serves as conjugating elements between $g_1^j$ and $g_2^{-j}$. It is easy to show that there are $z_{\nu_1^j} = z_{\nu_2^j}$ of these elements, and thus that $|\widetilde{T}_{\nu_1,\nu_2}(j)| = z_{\nu_1^j} $. 
%Consequently, we have in the general case $|\widetilde{T}_{\nu_1,\nu_2}(m)| = z_{\nu_1^m}  \delta_{\nu_1^m,\nu_2^m}$, which finally yields
%\[
%|T_{\nu_1, \nu_2}(j)| = \sum_{m|j} \mu(j/m) z_{\nu_1^m} \delta_{\nu_1^m,\nu_2^m}
%\]
%and concludes the proof.
\end{proof}

\section{Estimation of \texorpdfstring{$|T_{\nu_1,\nu_2}(j)|$}{|T(v₁,v₂)(j)|} and Preliminary Results}

\label{sec: Aux}

In light of \Cref{cor - two-side}, we naturally need a good estimate for $|T_{\nu_1, \nu_2}(j)|$ in order to get a good asymptotic of $T(n, k)$. To this end, we prove some estimates on $z_{\nu^m}$. Note that we trivially have $z_{\nu^m} \geq z_\nu$ for all $m\in \mathbb{N}$ and all integer partition $\nu$. %, since 
%\[
%z_{\nu^m} = |\{ \sigma\in\Sym{n} : g^m \sigma g^m=\sigma \}|,
%]
%where $g$ has cycle type $\nu$, and $g^1 \sigma g^1=\sigma$ implies that $g^m \sigma g^m=\sigma$ for $m> 1$.
However, we need a reverse inequality, and this is the content of the following result, which may be of independent interest.

\begin{proposition}\label{lem - ratio}
    Let $p$ be a prime number. Then we have $z_{\nu^p} \leq n!p^{\frac{1-n}{p}}z_\nu$.
\end{proposition}
\begin{proof}
    We proceed by strong induction on the size $n$ of the partition $\nu$. The inequality is trivially satisfied for $n=1$. Hence, suppose that $n\geq 2$, let $\nu = (1^{m_1} 2^{m_2} \dots n^{m_n})$, and write $\nu = (k,\tilde{\nu})$ where $\tilde{\nu} \vdash (n-k)$ and $k$ is the largest part of $\nu$, i.e., the largest $m_k \geq 1$. If $k<p$, then the primality of $p$ implies that $\nu^p=\nu$%. Since $n! p^{\frac{1-n}{p}} \geq 1$ for all $n\geq 1$ and $p\in\{2,3\}$, 
    , and the conclusion follows trivially. Hence, we suppose without loss of generality that $k\geq p$.  
    Now, a direct computation reveal that
    \[
    z_\nu = \frac{z_{\tilde{\nu}}}{k^{m_k-1} (m_k-1)!} \cdot k^{m_k}m_k! = z_{\tilde{\nu}} k m_k.
    \]
    Moreover, if $k$ is coprime with $p$, then we have
    \[
    z_{\nu^p} = \frac{z_{\tilde{\nu}^p}}{k^{m_k+pm_{pk}-1} (m_k+pm_{pk}-1)!} \cdot k^{m_k+pm_{pk}}(m_k+pm_{pk})! = z_{\tilde{\nu}^p} k (m_k+pm_{pk}).
    \]
    Therefore, since $km_k+pkm_{pk} \leq \sum_{j=1}^n jm_j = n$, and since $m_k\geq 1$ by construction, we have
    \[
    \frac{z_{\nu^p}}{z_\nu} = \frac{z_{\tilde{\nu}^p} (km_k+pkm_{pk})}{z_{\tilde{\nu}} k m_k} \leq \frac{z_{\tilde{\nu}^p}}{z_{\tilde{\nu}}} \cdot \frac{n}{k} \leq  (n-k)! p^{\frac{1-n+k}{p}} \cdot \frac{n}{k} =  \frac{p^{\frac{k}{p}}}{k!}\binom{n-1}{k-1}^{\!-1} n!p^{\frac{1-n}{p}}.
    \]
    Since $k\mapsto p^{k/p}/k!$ is decreasing for all $k\geq 1$ and $p\geq 1$, we then have
    \[
    \frac{z_{\nu^p}}{z_\nu} \leq \frac{p^{\frac{p}{p}}}{p!}\binom{n-1}{k-1}^{\!-1} n!p^{\frac{1-n}{p}} = \frac{1}{p!} \binom{n-1}{k-1}^{\!-1} n!p^{\frac{1-n}{p}} \leq n!p^{\frac{1-n}{p}},
    \]
    as desired. 
    Now, suppose that $k$ is a multiple of $p$. If $k/p$ is coprime with $p$, then we have
    \begin{align*}
        z_{\nu^p} &= \frac{(k/p)^{m_{k/p} + pm_{k}} (m_{k/p} + pm_{k})!}{(k/p)^{m_{k/p} + pm_{k}-p} (m_{k/p} + pm_{k}-p)!} z_{\tilde{\nu}^p} =  \Big(\frac{k}{p}\Big)^{p} \frac{(m_{k/p} + pm_{k})!}{(m_{k/p} + pm_{k}-p)!} z_{\tilde{\nu}^p}.
    \end{align*}
    Otherwise, $k/p$ is a multiple of $p$ and we have
    \begin{align*}
        z_{\nu^p} &= \frac{z_{\tilde{\nu}^p}}{(k/p)^{ pm_{k}-p} (pm_{k}-p)!} \cdot  \Big(\frac{k}{p}\Big)^{pm_{k}} (pm_{k})!  = z_{\tilde{\nu}^p} \Big(\frac{k}{p}\Big)^{p} \frac{(pm_{k})!}{(pm_{k}-p)!},
    \end{align*}
    which is identical to the first case if $m_{k/p}=0$. Hence, we proceed with the first case without any loss of generality. Using once again the fact that $\frac{k}{p} (m_{k/p} + pm_{k}) \leq n$ and $m_k\geq 1$ yield
    \begin{align*}
        \frac{z_{\nu^p}}{z_\nu} &= \frac{z_{\tilde{\nu}^p} \big(\frac{k}{p}\big)^{p} (m_{k/p} + pm_{k})\cdots(m_{k/p} + pm_{k}-p+1)}{z_{\tilde{\nu}} k m_k} \leq \frac{z_{\tilde{\nu}^p} (n-\frac{0\cdot k}{p})\cdots(n-\frac{(p-1)k}{p})}{z_{\tilde{\nu}} k} \\
        &\leq \frac{z_{\tilde{\nu}^p} \, n(n-1)\cdots (n-p+1)}{z_{\tilde{\nu}} k} = \frac{z_{\tilde{\nu}^p} \, n!}{z_{\tilde{\nu}} k(n-p)!}  \leq \frac{p^{\frac{k}{p}}(n-k)!}{k(n-p)!} \cdot n! p^{\frac{1-n}{p}}.
    \end{align*}
    Lastly, use once more the fact that $k\mapsto p^{k/p}/k!$ is decreasing to obtain
    \begin{align*}
        \frac{z_{\nu^p}}{z_\nu} &\leq \frac{p^{\frac{p}{p}} (k-1)!(n-k)!}{p!(n-p)!} \cdot n! p^{\frac{1-n}{p}} = \frac{1}{k} \binom{n}{p}\binom{n}{k}^{\!-1} n! p^{\frac{1-n}{p}} \leq n! p^{\frac{1-n}{p}}.
    \end{align*}    
\end{proof}

Let us now bring back our attention on the estimation of $T_{\nu_1,\nu_2}(j)$. It is instructive to first consider the simple case when $j=1$. In that case, \Cref{cor - two-side} ensures that
\begin{align*}
    |T_{\nu_1,\nu_2}(1)| = \sum_{m|1} \mu(1/m) z_{\nu_1^m} \delta_{\nu_1^m,\nu_2^m} = z_{\nu_1} \delta_{\nu_1,\nu_2}.
\end{align*}
The following lemma expand on this by providing an upper bound on $|T_{\nu_1, \nu_2}(j)|$ in the general case, which shall prove to be crucial in the proof of \Cref{thm:Asymptotic}.

\begin{lemma}\label{lem - partition}
    Let $\nu_1$ and $\nu_2$ be two partitions of $n$, and let $T_{\nu_1,\nu_2}(j)$ be the set defined in \eqref{eq - T}. Then $|T_{\nu_1,\nu_2}(1)| = z_{\nu_1} \delta_{\nu_1,\nu_2}$. Moreover,
    \[
    |T_{\nu_1,\nu_2}(j)| \leq \begin{cases}
        j^{\frac{1-n}{j}}n! \sqrt{z_{\nu_1} z_{\nu_2}} \quad&\text{if } j \text{ is prime}, \\
        % 3^{\frac{1-n}{3}}n! \sqrt{z_{\nu_1} z_{\nu_2}} \quad&\text{if }~ j=3, \\
        \,\,n!  \quad&\text{otherwise}.
    \end{cases}
    \]
\end{lemma}
\begin{proof}
The bound $|T_{\nu_1, \nu_2}(j)|\leq n!$ is trivial since $T_{\nu_1,\nu_2}(j) \subseteq \Sym{n}$. The case $|T_{\nu_1, \nu_2}(1)|$ is already discussed. Now, recall that $|\widetilde{T}_{\nu_1,\nu_2}(p)| = z_{[\nu_1]^p} \delta_{\nu_1^p,\nu_2^p}$ and that $T_{\nu_1,\nu_2}(p) \subseteq \widetilde{T}_{\nu_1,\nu_2}(p)$. Hence
    \begin{align*}
        |T_{\nu_1,\nu_2}(p)| \leq |\widetilde{T}_{\nu_1,\nu_2}(p)| \leq \min\{z_{\nu_1^p}, z_{\nu_2^p}\} \leq \sqrt{z_{\nu_1^p} z_{\nu_2^p}} \leq n!p^{\frac{1-n}{p}} \sqrt{z_{\nu_1} z_{\nu_2}},
    \end{align*}
    where the last inequality follows from \Cref{lem - ratio}. 
\end{proof}

Lastly, to conclude this section, we prove the following results which provides an asymptotic expression for the quantity $X_n = \sum_{\nu\vdash n}z_{\nu}-n!$. The short proof relies heavily on the work that was done in \cite{Geloun2022}.

\begin{lemma}
\label{lemma: asymptotic for sum of z}
Let $X_n = \sum_{\nu\vdash n}z_{\nu}-n!$. Then, 
\[ X_n = n! \!\left(\frac{2}{n^2}+ o(n^{-2})\right) = (2+o(1))\frac{n!}{n^2}\;.\]
\end{lemma}
\begin{proof}
From~\cite[Theorem 3]{Geloun2022} (more precisely, see Equation $(2.59)$ with $K=2$), we have
\[ 
 X_n = n! \!\left(\frac{S_{3;1;2}-n!}{n!}+o(n^{-2})\right).
\]
However, by~\cite[Equation (2.20)]{Geloun2022} we have $S_{3;1;2}-n!= \frac{2 n!}{n(n-1)}$ which completes the proof.
\end{proof}

% \begin{remark}
%     Using the upper bounds in the proofs of \cite[Lemma 2 and Theorem 2]{Geloun2022}, and refining them when $K=2$, one can show that 
%     \[
%     X_n \leq \frac{3n!}{n^{2}}, \qquad n\geq 18.
%     \]
%     Combined with direct computation for the value of $X_n$ for $n\leq 17$, we find that
%     \[
%     X_n \leq \frac{18n!}{n^{2}}, \qquad n\geq 1.
%     \]
% \end{remark}

%From~\cite[Theorem 3]{Geloun2022} (more precisely, see Equation $(2.59)$ with $K=2$), we have that when $n$ is even,
% \[ 
% n^2 \frac{X_n+n! - S_{3;1;2}}{n!} \leq \mathcal{R}_{n,2} + \frac{n^2}{n!} \text{Sym}[2^{n/2}] P_1(n),
% \]
% where $\text{Sym}[2^{n/2}] = z_{(2^{n/2})} = 2^{n/2} (n/2)!$, $P_1(n) = p(n)-p(n-1)$, and
% \[
% \mathcal{R}_{n,2} = n^2 \sum_{k=3}^n \frac{1}{n(n-1)\cdots (n-k+1)} \sum_{\nu\vdash k, \nu_1=0} z_\nu.
% \]
% Using the upper bounds in the proofs of \cite{Geloun2022}, specifically of Lemma 2 and Theorem 2, and refining them when $K=2$, we may find that for all $n\geq 16$, we have
% % \[
% % \mathcal{R}_{n,2} < \frac{n}{(n-1)(n-2)}\left(9+\frac{3}{n-4}+\frac{32}{n-3}+\frac{264}{(n-3)(n-5)}\right).
% % \]
% However, by~\cite[Equation (2.20)]{Geloun2022} we have $S_{3;1;2}-n!= \frac{2\cdot n!}{n(n-1)}$.  Hence, we find altogether that
% \begin{align*}
%     &n^2 \frac{X_n+n! - S_{3;1;2}}{n!} \\
%     &< \frac{n}{(n-1)(n-2)}\left(9+\frac{3}{n-4}+\frac{32}{n-3}+\frac{264}{(n-3)(n-5)}\right) + \frac{2^{n/2} (n/2)!n^2}{n!}(p(n)-p(n-1)),
% \end{align*}

\section{A Sharp Upper Bound: Proof of Theorem~\ref{thm:Asymptotic}}
\label{sec - sharp_upperbound}

\subsection{General outline}

The goal of this final section is to prove our main result, namely \Cref{thm:Asymptotic}. Recall from \Cref{cor - two-side} that
\[   T(n, k) = \sum_{(\nu_1, \nu_2)\vdash n}\frac{1}{z_{\nu_1}z_{\nu_2}}\sum_{\lambda\vdash k} \prod_{j=1}^{k}\binom{|T_{\nu_1, \nu_2}(j)|/j}{m_j(\lambda)}\;.\]
Consider the case $\nu_1=\nu_2=(1^n)$. It is easily verified that $z_{\nu_1}=z_{\nu_2}=|T_{\nu_1,\nu_2}(1)|=n!$ and $|T_{\nu_1,\nu_2}(j)|=0$ if $j\neq 1$. In particular, the product term in the equation above is zero for any $\lambda\vdash k$ such that $m_\ell(\lambda) \neq 0$ for some $\ell>1$. Thus, the only partition $\lambda \vdash k$ yielding a nonzero product is $\lambda= (1^k)$, that is when $m_1(\lambda) = k$. 
In this case, we find
\[
\frac{1}{z_{(1^n)} z_{(1^n)}} \sum_{\lambda\vdash k}\prod_{j=1}^{k}\binom{|T_{(1^n),(1^n)}(j)|/j}{m_j(\lambda)} = \frac{1}{n!^2}\binom{n!}{k} =  \Lambda_n(k) \;.
\]
In particular, we get 
\[T(n, k) = \Lambda_n(k) \!\left(1+ \sum_{\lambda\vdash k }R(n, k, \lambda)\right),\]
where 
\begin{equation}
\label{eqn: Target}
  R(n, k, \lambda) := \frac{1}{\Lambda_n(k)} \!\sum_{\substack{\nu_1,\nu_2\vdash n \\ (\nu_1,\nu_2) \neq (1^n)^2}}\! \frac{1}{z_{\nu_1} z_{\nu_2}}\prod_{j=1}^{k}\binom{|T_{\nu_1,\nu_2}(j)|/j}{m_j(\lambda)}\;, 
\end{equation}
and we write $R(n, k)=\sum_{\lambda\vdash k}R(n, k, \lambda)$. 
\Cref{thm:Asymptotic} follows if we can show that $R(n, k)\to 0$, uniformly in $k\in [3, n!-3]$, as $n\to \infty$. This is precisely the content of our next proposition.
\begin{proposition}[Uniform decay of $R(n, k)$]
\label{prop: Uniform decay} 
Let $R(n, k)$ be as above. 
There is a universal constant $C>0$ (independent of both $n$ and $k$) such that
\begin{align*}
    R(n, k) &\leq \frac{C}{n^2},
\end{align*}
for all $n\geq 1$ and all $3\leq k\leq n!-3$. 
\end{proposition}
This completes the proof of \Cref{thm:Asymptotic}. As the proof of \Cref{prop: Uniform decay} relies on several lemmas and several cases, we give the proof in the next section. 
 
\subsection{Proof of Proposition~\ref{prop: Uniform decay}} 
This section is devoted to the proof of this proposition. Since $R(n,k)=R(n,n!-k)$, we assume without loss of generality that $3\leq k \leq n!/2$. We split the proof of \Cref{prop: Uniform decay} into several propositions. \Cref{prop - k=3to5} deals with the case $k=3, 4$ and $5$, and \Cref{prop: k larger than 5} deals with the case $k\geq 6$. 
In the following, we repeatedly use the well-known estimates
\begin{gather}\label{eq - estimations}
    \frac{n^k}{k^k} \leq \binom{n}{k} \leq \frac{n^k}{k!},  \qquad 1\leq k\leq n.
\end{gather}

\subsubsection{The cases \texorpdfstring{$\bm{k=3,4}$ and 5}{k=3, 4, and 5}}
\begin{proposition}[The cases \texorpdfstring{$\bm{k=3,4}$ and 5}{k=3, 4, and 5}]
\label{prop - k=3to5}
Suppose that $k\in\{3,4,5\}$. There exists a constant $C>0$, independant of $k$, such that for all $n\geq 1$,
    \[ R(n, k)\leq \frac{C}{n^2}.\]
\end{proposition}
\begin{proof}
    Since there are only a finite number of partitions of 3, 4 and 5, it suffices to show that $R(n, k, \lambda)\leq C/n^2$ for some $C$ for each such partitions $\lambda$. %We treat each $\lambda$ separately. 
    In the following, we use the fact that $n!^{k-2}/\Lambda_n(k) = \frac{k!n!^{k-1}}{(n!-1)\cdots (n!-k+1)}\leq (2k)^k\binom{2k}{k}^{\!-1} =:C_k$ for all $n!\geq 2k$.
    
    \smallskip
    \noindent\textbf{Case 1: $\bm{k=3}$.}
    \vspace{-6pt}
    \begin{itemize}[leftmargin=22pt]
        \item \textbf{Subcase 1: $\bm{\lambda = (1^3)}$.}
        Using \Cref{lem - partition} and equation \cref{eq - estimations}, we obtain
        \begin{align*}
            R(n, 3, \lambda) \leq \frac{1}{6\Lambda_n(3)} \!\!\!\sum_{\substack{\nu_1,\nu_2\vdash n \\ (\nu_1,\nu_2) \neq (1^n)^2}}\!\!\! \frac{|T_{\nu_1,\nu_2}(1)|^3}{z_{\nu_1} z_{\nu_2}} \leq \frac{2}{n!}\sum_{(1^n)\neq \nu \vdash n}\! z_{\nu}.% \leq \frac{36}{n^2}
        \end{align*}
        The conclusion now follows from \Cref{lemma: asymptotic for sum of z}.

        % \begin{align*}
        %     R(n, 3, \lambda) &= \frac{1}{6\Lambda_n(3)} \sum_{(1^n)\neq \nu\vdash n} \frac{(z_{\nu}-1)(z_{\nu}-2)}{z_{\nu}} = \frac{1}{6\Lambda_n(3)} \sum_{(1^n)\neq \nu\vdash n} (z_\nu-3+\frac{2}{z_\nu}) \\
        %     %
        %     &\leq \frac{(2+\varepsilon)n!}{6n^2\Lambda_n(3)}+\frac{1}{6\Lambda_n(3)}\left( 2-\frac{2}{n!}-3p(n)+3 \right)
        % \end{align*}

        \item \textbf{Subcase 2: $\bm{\lambda = (1^1,2^1)}$.} 
        Once again we use \Cref{lem - partition} to obtain
        \begin{align*}
             R(n, 3, \lambda) = \frac{1}{2\Lambda_n(3)} \!\!\!\sum_{\substack{\nu_1,\nu_2\vdash n \\ (\nu_1,\nu_2) \neq (1^n)^2}}\!\!\! \frac{|T_{\nu_1,\nu_2}(1) | |T_{\nu_1,\nu_2}(2)|}{z_{\nu_1} z_{\nu_2}}
             \leq \frac{6}{2^{\frac{n-1}{2}}} \sum_{\nu \vdash n } 1 = \frac{6p(n)}{2^{\frac{n-1}{2}}}.
        \end{align*}

        \item \textbf{Subcase 3: $\bm{\lambda = (3^1)}$.} 
        We use one last time \Cref{lem - partition} to obtain
        \begin{align*}
           R(n, 3, \lambda) = \frac{1}{3\Lambda_n(3)} \!\!\!\sum_{\substack{\nu_1,\nu_2\vdash n \\ (\nu_1,\nu_2) \neq (1^n)^2}}\!\!\! \frac{|T_{\nu_1,\nu_2}(3)|}{z_{\nu_1} z_{\nu_2}}
            &\leq  \frac{4}{3^{\frac{n-1}{3}}} \sum_{\nu_1,\nu_2\vdash n} \!1  =  \frac{4p^2(n)}{3^{\frac{n-1}{3}}}\;.
        \end{align*}  
    \end{itemize}
    % The conclusion follows by noting that $\smash{\frac{p(n)}{\Gamma(n/2+1)} \leq \frac{p(n)^2}{\Gamma(n/3+1)}\leq \frac{C}{n^2}}$. 
    \vspace{-6pt}
    \noindent\textbf{Case 2: $\bm{k=4}$.}
    \vspace{-6pt}
    \begin{itemize}[leftmargin=22pt]
        \item \textbf{Subcase 1: $\bm{\lambda=(2^2)}$.}  
        We find by \Cref{lem - partition} and the upper bound of \eqref{eq - estimations} that
        \begin{align*}
           R(n, 4, \lambda) &\leq  
            \frac{C_4}{n!^2} \!\sum_{\nu_1,\nu_2\vdash n}\! \frac{|T_{\nu_1,\nu_2}(2)|^2}{z_{\nu_1} z_{\nu_2}} \leq
            \frac{C_4}{2^{n-1}} \!\sum_{\nu_1,\nu_2\vdash n}\! \frac{z_{\nu_1} z_{\nu_2}}{z_{\nu_1} z_{\nu_2}}= \frac{C_4 p^2(n)}{2^{n-1}}.
        \end{align*}

        \item \textbf{Subcase 2: $\bm{\lambda = (4^1)}$.} Since $\sum_{\nu\vdash n}\frac{1}{z_{\nu}}=1$, we have
        \begin{align*}
           R(n, 4, \lambda)
            &\leq \frac{C_4}{n!^2} \!\!\sum_{\substack{\nu_1,\nu_2\vdash n \\ (\nu_1,\nu_2) \neq (1^n)^2}}\!\!\! \frac{|T_{\nu_1,\nu_2}(4)|}{z_{\nu_1} z_{\nu_2}} \leq \frac{C_4}{n!^2}  \sum_{\nu_1,\nu_2\vdash n } \frac{n!}{z_{\nu_1} z_{\nu_2}} =\frac{C_4}{n!}.
        \end{align*}

        \item \textbf{Subcase 3: $\bm{\lambda = (1^1, 3^1)}$}
        Using \Cref{lem - partition}, we obtain 
        \[R(n, 4, \lambda) = \frac{1}{3\Lambda_n(4)} \!\!\!\sum_{\substack{\nu_1,\nu_2\vdash n \\ (\nu_1,\nu_2) \neq (1^n)^2}}\!\!\! \frac{|T_{\nu_1,\nu_2}(1) | |T_{\nu_1,\nu_2}(3)|}{z_{\nu_1} z_{\nu_2}} \leq \frac{C_4}{n!^2}\sum_{\nu\neq (1^n)}\frac{z_{\nu}\cdot n!}{z_{\nu}^2}\leq \frac{C_4}{n!}\;.\]
 
        \item \textbf{Subcase 4: $\bm{\lambda = (1^2, 2^1)}$}
        Using \Cref{lem - partition}, we obtain 
        \[ R(n, 4, \lambda) \leq \frac{1}{4\Lambda_n(4)} \!\!\!\sum_{\substack{\nu_1,\nu_2\vdash n \\ (\nu_1,\nu_2) \neq (1^n)^2}}\!\!\! \frac{|T_{\nu_1,\nu_2}(1) |^2 |T_{\nu_1,\nu_2}(2)|}{z_{\nu_1} z_{\nu_2}} \leq  \frac{C_4}{n!^2} \sum_{\nu\vdash n} n! \leq \frac{C_4 p(n)}{n!}\;.\]
 
        \item \textbf{Subcase 5: $\bm{\lambda = (1^4)}$} 
        \Cref{lem - partition} and \Cref{lemma: asymptotic for sum of z} yield together
        \[
        R(n, 4, \lambda) \leq  \frac{1}{24\Lambda_n(4)} \!\!\!\sum_{\substack{\nu_1,\nu_2\vdash n \\ (\nu_1,\nu_2) \neq (1^n)^2}}\!\!\!\! \frac{|T_{\nu_1,\nu_2}(1) |^4}{z_{\nu_1} z_{\nu_2}}  \leq \frac{C_4}{n!^2} \!\!\sum_{(1^n)\neq\nu \vdash n} \!\!z_{\nu}^2 \leq    \frac{C_4}{n!^2} \left(  \sum_{(1^n)\neq\nu \vdash n } \!\!z_{\nu}\right)^{\!\smash{2}} \!\leq \frac{C_4}{n^4} .
        \]
    \end{itemize}

    \vspace{-6pt}
    \noindent\textbf{Case 3: $\bm{k=5}$.}
    \vspace{-6pt}
    \begin{itemize}[leftmargin=22pt]
        \item \textbf{Subcase 1: $\bm{\lambda=(1^5)}$ or $\bm{(1^3,2^1)}$.}  
        \Cref{lem - partition} and the upper bound of \eqref{eq - estimations} yield
        \begin{align*}
           R(n, 5, \lambda) &\leq  
            \frac{C_5}{n!^3} \!\sum_{\substack{\nu_1,\nu_2\vdash n \\ (\nu_1,\nu_2) \neq (1^n)^2}}\! \frac{|T_{\nu_1,\nu_2}(1)|^{5-2m} |T_{\nu_1,\nu_2}(2)|^{m}}{z_{\nu_1} z_{\nu_2}} \leq \frac{C_5}{n!^3} \sum_{(1^n)\neq\nu\vdash n} z_\nu^{3-2m} n!^m \\
            &\leq  \frac{C_5}{n!^{3-m}} \left(\sum_{(1^n)\neq\nu\vdash n} z_\nu \right)^{\!3-2m} \leq \frac{C_5'}{n!^{m} n^{6-4m}},
        \end{align*}
        where the last inequality follows from \Cref{lemma: asymptotic for sum of z} and where $m=0$ or $1$.

        \item \textbf{Subcase 2: $\bm{\lambda = (1^2, 3^1)}$ or $\bm{(1^1, 2^2)}$.}
        Using \Cref{lem - partition} and the upper bound of \eqref{eq - estimations}, we obtain for $m=1$ or $m=2$
        \[R(n, 5, \lambda) \leq \frac{C_5}{n!^3} \!\!\!\sum_{\substack{\nu_1,\nu_2\vdash n \\ (\nu_1,\nu_2) \neq (1^n)^2}}\!\!\! \frac{|T_{\nu_1,\nu_2}(1) |^m |T_{\nu_1,\nu_2}(\frac{5-m}{3-m})|^{3-m}}{z_{\nu_1} z_{\nu_2}} \leq \frac{C_5}{n!^3}\sum_{ \nu \vdash n} \frac{n!^{3-m}}{z_\nu^{2-m}}  \leq \frac{C_5 p(n)}{n!^m}\;.\]
 
        \item \textbf{Subcase 3: $\bm{\lambda = (1^1,4^1),  (2^1,3^1)}$ or $\bm{(5^1)}$.} 
        Using \Cref{lem - partition}, we may bound each $|T_{\nu_1,\nu_2}(j)|$ by $n!$ to obtain
        \vspace{-4pt}
        \[
        R(n, 5, \lambda) \leq  \frac{C_5}{n!^3} \!\!\!\smash{\sum_{\substack{\nu_1,\nu_2\vdash n \\ (\nu_1,\nu_2) \neq (1^n)^2}}}\!\!\! \frac{n!^2 }{z_{\nu_1} z_{\nu_2}}  \leq  \frac{C_5}{n!} \;.\]
    \end{itemize}
\end{proof}

\subsubsection{The case \texorpdfstring{$\bm{k\geq 6}$}{k≥6}}
We now deal with \Cref{prop: Uniform decay} when $k\geq 6$. We begin with some preliminary facts that will be repeatedly used in the proof. We note that a simple counting argument shows that
\[
\binom{n}{m}^{\!k} \leq \binom{kn}{km}, \qquad k\geq 1.
\]
%Indeed, the left-hand side counts the number of ways to select $m$ objects among $k$ sets of $n$ objects, while the right-hand side counts the number of ways to select $km$ objects among $kn$ objects. Since the $kn$ objects can be interpreted as the union of the $k$ sets of $n$ objects from the left-hand side, it is clear that each combination on the left-hand side is also counted on the right-hand side, and thus that the inequality holds. 
We also recall the following elementary identity, often known as the generalized Chu--Vandermonde identity~\cite[Theorem 5.2]{Vignat15},
%\textcolor{red}{Maybe add a citation and drop the explanation}\textcolor{blue}{I tried to find one for around 3 hours and couldn't find exactly this identity, except on Wikipedia and on an old unpublished paper on arXiv. If you find one I'll be glad to remove the explanation.}\textcolor{red}{I have a added a reference for the time being. I am not very happy with it because it doesn't state the exact identity explicitly. If you can find    Combinatorial Idenitities' by John Riordan, it maybe worth checking for the exact statement. At any rate I have also added a one line explanation.}, which states that
\begin{equation*}\label{eq - Chu-Vandermonde}
    \sum _{k_{1}+\cdots +k_{p}=m}\binom{n_{1}}{k_{1}}\binom{n_{2}}{k_{2}}\cdots \binom{n_{p}}{k_{p}}=\binom{n_{1}+\dots + n_{p}}{m}.
\end{equation*}
The above identity is immediate by noticing that both sides count the coefficient of $x^{m}$ in $\prod_{j=1}^{p}(1+x)^{n_i}=(1+x)^{\sum_{i=1}^{p}n_i}$. 
%The identity holds since on the left-hand side, we choose $k_1$ elements from a set of $n_1$ items, then $k_2$ from another set of $n_2$ items, and so on, until we've chosen a total of $m = k_1 + k_2 + \cdots + k_p$ elements from $p$ distinct sets, which is the same process as selecting $m$  elements from the combined pool of $n_1 + n_2 + \cdots + n_p$ items, i.e., exactly what is being done on the right-hand side. 
We also record the pair of inequalities 
\begin{equation}\label{eq - binom_even}
    \frac{2^{2n}}{2\sqrt{n}} \leq \binom{2n}{n} \leq \frac{2^{2n}}{\sqrt{\pi n}},\qquad n\geq 1,
\end{equation}
which appears in \cite{Watson1958}. See \cite{Qi2010} for a thorough survey of inequalities of this kind. Finally, we need the following lemma.
\begin{lemma}
\label{lemma: log-convex}
    Let $n\in\mathbb{N}$. The function $k\mapsto \binom{n}{k}$ is log-concave in the interval $1\leq k \leq n-1$.
\end{lemma}
\begin{proof}
    This is equivalent to proving that $\smash{\binom{n}{k}^{\!2} \geq \binom{n}{k-1}\binom{n}{k+1}}$, which holds since
    \[
    \frac{\binom{n}{k}^{\smash{\!2}}}{\binom{n}{k-1}\binom{n}{k+1}} = \frac{n-k+1}{n-k} \cdot \frac{k+1}{k} \geq 1.
    \]
\end{proof}

We can now state and prove the main result of this section.
\begin{proposition}\label{prop: k larger than 5}
    Let $k\geq 6$. There exists a constant $C>0$ and a positive integer $N$, both independant of $k$, such that for all $n\geq N$,
    \[ R(n, k)\leq \frac{C}{n^2}\;.\]
\end{proposition}

\begin{proof}
Recall that 
\begin{align*}
    R(n,k) % \frac{1}{\Lambda_n(k)} \!\!\sum_{\substack{\nu_1,\nu_2\vdash n \\ (\nu_1,\nu_2) \neq (1^n)^2}}\! \frac{1}{z_{\nu_1} z_{\nu_2}} \sum_{\lambda\vdash k}\prod_{j=1}^{k}\binom{|T_{\nu_1,\nu_2}(j)|/j}{m_j(\lambda)}\\
    &= \sum_{\lambda\vdash k} \frac{1}{\Lambda_n(k)} \!\!\sum_{\substack{\nu_1,\nu_2\vdash n \\ (\nu_1,\nu_2) \neq (1^n)^2}}\! \frac{1}{z_{\nu_1} z_{\nu_2}} \prod_{j=1}^{k}\binom{|T_{\nu_1,\nu_2}(j)|/j}{m_j(\lambda)}\;.
\end{align*}
Note that as soon as $m_1(\lambda)\neq 0$, we must have $\nu_1=\nu_2$ for the product term to be non-zero since $|T_{\nu_1,\nu_2}(1)|= \delta_{\nu_1,\nu_2} z_{\nu_1}$ by \Cref{lem - partition}. Let us define 
\begin{align*}
    R_0(n, k) &:= \sum_{\substack{\lambda\vdash k \\m_1(\lambda)=0}} \frac{1}{\Lambda_n(k)} \!\!\sum_{\substack{\nu_1,\nu_2\vdash n \\ (\nu_1,\nu_2) \neq (1^n)^2}}\! \frac{1}{z_{\nu_1} z_{\nu_2}} \prod_{j=2}^{k}\binom{|T_{\nu_1,\nu_2}(j)|/j}{m_j(\lambda)}\\
    R_1(n, k) &:= \sum_{\substack{\lambda\vdash k\\ m_1(\lambda)\neq 0}} \frac{1}{\Lambda_n(k)} \sum_{(1^n)\neq\nu\vdash n} \frac{1}{z_{\nu}^2} \prod_{j=1}^{k}\binom{|T_{\nu,\nu}(j)|/j}{m_j(\lambda)}\;,
\end{align*}
so that we have $R(n,k)=R_0(n,k)+R_1(n,k)$. To control $R_{0}(n, k)$, observe that 
\begin{align*}
    R_0(n, k) &\leq \frac{1}{\Lambda_n(k)} \!\!\sum_{\substack{\nu_1,\nu_2\vdash n \\ (\nu_1,\nu_2) \neq (1^n)^2}}\! \frac{1}{z_{\nu_1} z_{\nu_2}} \sum_{\substack{\lambda\vdash k \\m_1(\lambda)=0}} \prod_{j=2}^{k}\binom{|T_{\nu_1,\nu_2}(j)|}{jm_j(\lambda)}^{\!1/2} \\
    &\leq  \frac{1}{\Lambda_n(k)} \!\!\sum_{\substack{\nu_1,\nu_2\vdash n \\ (\nu_1,\nu_2) \neq (1^n)^2}}\! \frac{1}{z_{\nu_1} z_{\nu_2}} \binom{n!}{k}^{\!1/2} \leq n!^2 \binom{n!}{k}^{\!-1/2} \\
    &\leq n!^2 \binom{n!}{6}^{\smash{\!-1/2}} \leq \frac{30}{n!}\leq \frac{1}{n^2}\;. 
\end{align*}
Here, we used the fact that $\binom{n}{m}^{\smash{j}}\leq \binom{jn}{jm}$ for any $j\geq 1$ in the first inequality. In the second inequality, we used the generalized Chu-Vandermonde inequality and the fact that $\sum_{j\geq 1} |T_{\nu_1,\nu_2}(j)| = n!$. The third inequality uses the fact that $\sum z_{\nu}^{-1}=1$, and the fourth follows by noting that $k\mapsto\binom{n!}{k}$ is increasing for all $6\leq k \leq n!/2$. The last two inequalities hold for all $n\geq 5$ by direct verification. 
On the other hand, we have
\begin{align*}
    R_1(n, k) &\leq \frac{1}{\Lambda_n(k)} \sum_{(1^n)\neq\nu\vdash n } \frac{1}{z_{\nu}^2} \sum_{\lambda\vdash k}\binom{|T_{\nu, \nu}(1)|}{m_1}^{\!1/2}  \cdot \prod_{j=1}^{k}\binom{|T_{\nu,\nu}(j)|}{jm_j}^{\!1/2} \\
    &\leq  \frac{1}{\Lambda_n(k)} \sum_{(1^n)\neq\nu\vdash n} \frac{1}{z_{\nu}^2} \left(\sum_{\lambda\vdash k}\binom{|T_{\nu, \nu}(1)|}{m_1}\right)^{\!1/2} \left( \sum_{\lambda\vdash k}\prod_{j=1}^{k}\binom{|T_{\nu,\nu}(j)|}{jm_j}\right)^{\!1/2} \\
    &\leq n!^2 \binom{n!}{k}^{\!-\frac{1}{2}} \!\!\!\sum_{(1^n)\neq\nu\vdash n } \frac{1}{z_{\nu}^2} \left( \sum_{\lambda\vdash k} \binom{z_{\nu}}{m_1} \right)^{\!\frac{1}{2}} \leq \sqrt{p(k)}n!^2 \binom{n!}{k}^{\!-\frac{1}{2}} \!\!\!\sum_{(1^n)\neq\nu\vdash n } \frac{1}{z_{\nu}^2} \left( \sum_{m=1}^{k} \binom{z_{\nu}}{m} \right)^{\!\frac{1}{2}},
\end{align*}
where the second inequality follows from Cauchy--Schwarz and the third inequality uses the Chu--Vandermonde inequality. First, suppose that $z_\nu \leq k$. Since $\binom{z_{\nu}}{m}=0$ for all $m>z_\nu$, we have
\[
\sum_{m=1}^{k} \binom{z_{\nu}}{m}  = \sum_{m=1}^{z_\nu} \binom{z_{\nu}}{m} \leq 2^{z_\nu}.
\]
On the other hand for $z_\nu > k$ we have
\[
\sum_{m=1}^{k} \binom{z_{\nu}}{m} \leq \sum_{m=1}^{k} \frac{z_{\nu}^m}{m!} = \sum_{m=1}^{k} \frac{k^m}{m!}  \cdot \frac{z_\nu^m}{k^m} \leq e^k \cdot \frac{z_\nu^k}{k^k}.
\]
Combining this with \Cref{lemma: asymptotic for sum of z}, we find that for all $\varepsilon>0$, there is some $N_\varepsilon\in\mathbb{N}$ such that for all $n\geq N_\varepsilon$ we have
\begin{align*}
    \sum_{(1^n)\neq\nu\vdash n } \frac{1}{z_{\nu}^2} \left( \sum_{m=1}^{k} \binom{z_{\nu}}{m} \right)^{\!\frac{1}{2}} &\leq \sum_{\substack{(1^n)\neq\nu\vdash n \\ z_\nu \leq k}} \frac{2^{z_\nu/2}}{z_{\nu}^2} + \frac{e^{\frac{k}{2}}} {k^{\frac{k}{2}}} \sum_{\substack{(1^n)\neq\nu\vdash n \\ z_\nu > k}} z_\nu^{\frac{k}{2}-2} \leq \sum_{\nu\vdash n} \frac{2^{k/2}}{k^2} + \frac{e^{\frac{k}{2}}} {k^{\frac{k}{2}}} \sum_{(1^n)\neq\nu\vdash n} z_\nu^{\frac{k}{2}-2} \\
    &\leq \frac{2^{\frac{k}{2}} p(n)}{k^2} + \frac{e^{\frac{k}{2}} }{k^{\frac{k}{2}} } \left(\sum_{(1^n)\neq\nu\vdash n} z_\nu \right)^{\!\!\frac{k}{2}-2} \!\!\leq \frac{2^{\frac{k}{2}}  p(n)}{k^2} + \frac{e^{\frac{k}{2}} }{k^{\frac{k}{2}} } \left(\frac{(2+\varepsilon)n!}{n^2} \right)^{\!\frac{k}{2}-2}\!.
\end{align*}

Let us further define 
\begin{align*}
    R_1'(n,k) &:= \sqrt{p(k)}n!^2 \binom{n!}{k}^{\!-\frac{1}{2}} \frac{2^{k/2} p(n)}{k^2}, \\
    R_1''(n,k) &:=  \sqrt{p(k)}n!^2 \binom{n!}{k}^{\!-\frac{1}{2}} \frac{e^{k/2}}{k^{k/2}} \left(\frac{(2+\varepsilon)n!}{n^2} \right)^{\!\frac{k}{2}-2}.
\end{align*}
Then $R_1(n,k) \leq  R_1'(n,k) +  R_1''(n,k)$. We now treat each term independently. Observe that for any $\alpha>0$, we have $\sqrt{x}\le\alpha x+\frac{1}{4\alpha}$. Since $p(k) \leq \frac{\exp({\pi\sqrt{2k/3}})}{4\sqrt{3}k}$, it follows that
\[
\sqrt{p(k)} \leq \frac{\exp\!\big(\pi\sqrt{k/6}\big)}{2\cdot 3^{1/4} \sqrt{k}} \leq \frac{\exp\!\big(\frac{\pi}{\sqrt{6}} \left(\alpha k+\frac{1}{4\alpha} \right)\!\big)}{2\cdot 3^{1/4} \sqrt{k}}.
\]
Therefore,
\begin{align*}
    R_1'(n,k) 
    &\leq p(n)n!^2 \cdot \binom{n!}{k}^{\!-\frac{1}{2}} \cdot \frac{\exp\!\big(\frac{\pi}{\sqrt{6}}\left(\alpha k+\frac{1}{4\alpha} \right)+\frac{k}{2}\big)}{2\cdot 3^{1/4} k^{5/2}} =: T_n(k) .
\end{align*}
\Cref{lemma: log-convex} ensures that $T_n(k)$ is a product of log-convex functions relative to $k$ in the interval $[6,n!/2]$. Hence, $T_n(k)$ is also log-convex and thus its maximum is realized either at $k=6$ or $k=n!/2$. We find, using \eqref{eq - estimations}, that
\begin{align*}
    T_n(6) &= p(n)n!^2  \binom{n!}{6}^{\!-\frac{1}{2}}  \frac{\exp\!\big(\frac{\pi}{\sqrt{6}}\left(6\alpha+\frac{1}{4\alpha} \right)+\frac{6}{2}\big)}{2\cdot 3^{1/4} 6^{5/2}} \leq C_\alpha \frac{p(n)}{n!}\;.
\end{align*}
At the second endpoint, we use \eqref{eq - binom_even} to obtain
\begin{align*}
    T_n\!\left(\frac{n!}{2}\right) = p(n)n!^2  \binom{n!}{n!/2}^{\!-\frac{1}{2}}  \frac{e^{\frac{\pi}{2\sqrt{6}}\left(\alpha n!+\frac{1}{2\alpha} \right)+\frac{n!}{4}}}{2\cdot 3^{1/4} (n!/2)^{5/2}} %= C_\alpha'  \frac{p(n)}{\sqrt{n!}} \binom{n!}{n!/2}^{\!-\frac{1}{2}} \exp\!\left( \left(\frac{\pi\alpha}{2\sqrt{6}} + \frac{1}{4}\right)\!n!\right) \\
    \leq C_\alpha' \frac{p(n)}{n!^{3/4}} \exp\!\left( \left(\frac{\pi\alpha}{2\sqrt{6}} + \frac{1}{4}-\frac{\log(2)}{2}\right)\!n!\right).
\end{align*}
Since $\alpha>0$ was arbitrary, we set $\alpha:=\sqrt{6}\left(\log\left(2\right)-\frac{1}{2}\right)/\pi >0$ to obtain 
\[
T_n\!\left(\frac{n!}{2}\right) \leq C_\alpha' \frac{p(n)}{n!^{3/4}}\;. 
\]
Therefore, $R_1'(n,k)$ always converges superexponentially to 0. 
Now, using \eqref{eq - estimations}, we find that $R_1''(n,k)$ satisfy
\vspace{-3pt}
\begin{align*}
    R_1''(n,k) \leq e^2\sqrt{p(k)}  \left( \frac{(2+\varepsilon)e}{n^2}\right)^{\!\smash{\frac{k}{2}-2}}.
\end{align*}
It is known that $p(k)$ is log-concave for all $k > 25$, so $\frac{p(k+1)}{p(k)} \leq \frac{p(k)}{p(k-1)}$ for all $k > 25$ \cite{Desalvo2015}. Hence, the maximum of $p(k+1)/p(k)$ occurs for some $1 \leq k \leq 25$. Computing these ratios for $5 \leq k < 25$ shows the maximum is $15/11$. Therefore, $k \mapsto p(k)/\beta^k$ is decreasing for $k \geq 6$ when $\beta \geq \frac{15}{11}$. Since $n \geq 5$, we have $\frac{n^2}{(2+\varepsilon)e} > \frac{15}{11}$ for all sufficiently small $\varepsilon > 0$, so we obtain
\[
R_1''(n,k)  \leq e^2\sqrt{p(6)}  \left( \frac{(2+\varepsilon)e}{n^2}\right)^{\!\frac{6}{2}-2} \!\!= \frac{(2+\varepsilon)e^3\sqrt{11}}{n^2},
\]
which completes the proof.
% Hence, we finally have
% \begin{align*}
%     R_1(n,k) &\leq R_1'(n,k)+R_1''(n,k) \leq C_\alpha' \frac{p(n)}{n!^{3/4}} + \frac{(2+\varepsilon)e^3\sqrt{11}}{n^2} \leq \frac{C}{n^2}
% \end{align*}
% for some constant $C$ which is \emph{independant of $k$}. This concludes the proof.
\end{proof}

% \section{Proof of Corollary~\ref{thm: bulk asymptotic}}
% \begin{proof}[Proof of Corollary~\ref{thm: bulk asymptotic}]
% By Theorem~\ref{thm:Asymptotic}, we have 
% \begin{align*}
%     B_n(x) = \frac{T(n, k_n(x))}{\Lambda_n} &= \frac{\binom{n!}{k_n(x)}}{\binom{n!}{k_n(0)}}(1+o_n(1))\;.
% \end{align*}
% The conclusion then follows directly from De Moivre--Laplace theorem, which states that
% \begin{align*}
%     2^{-n!}\binom{n!}{k_n(x)} =(1+o_n(1)) \frac{\exp\!\big(\!-\frac{(k_n(x)-n!/2)^2}{n!/2}\big)}{\sqrt{\pi n!/2}}.
% \end{align*}
% % % 
% % Therefore, it follows that
% % \[ B_n(x) = (1+o_n(1)) \exp\!\Big(\!-\frac{(k_n(x)-n!/2)^2}{n!/2}\Big)\to \exp(-x^2/2)\;.\]
% \end{proof}

\section{Open Questions}

\cite[Theorem 4.1]{diaconis2008fixed} gives the distribution of $\binom{[n]}{k}^{\!\sigma}$ for a uniformly random permutation $\sigma$, using~\ref{eqn:DFG_Formula}, and shows that the cycle counts $(c_j(\sigma))_{j \leq k}$ converge to independent Poisson random variables. This result raises the question of whether a similar description holds for random \emph{pairs} of permutations. In particular, analyzing the joint distribution of the statistics $|T_{g_1, g_2}(j)|$ when $(g_1, g_2)$ is chosen uniformly from $\Sym{n}\times \Sym{n}$ could shed light on the structure of their action.

\medskip
\noindent\textbf{Problem 1 (Random Pair Statistics).} Describe the joint distribution of the statistics $|T_{g_1, g_2}(j)|$ for a uniformly random pair of permutations $(g_1, g_2) \in \Sym{n}^2$. In particular, determine the distribution of the number of fixed points of the pair $(g_1, g_2)$ acting on $\Omega(n, k)$.
\medskip

Moreover, \Cref{cor - log-concavity} shows that \( B_n(x) \) becomes strictly log-concave as \( n \to \infty \), suggesting that log-concavity might eventually hold uniformly. Hence, this naturally raises the following question.

\medskip
\noindent\textbf{Problem 2 (Threshold for Log-Concavity).} Does there exist an integer \( N \) such that \( B_n(x) \) is log-concave for all \( n \geq N \)? If so, what is the smallest such \( N \)?
\medskip

The case \( k = 3 \) and \( \lambda = (1^3) \) in the proof of \Cref{prop - k=3to5} shows that the quadratic convergence rate of \( T(n,k)/\Lambda_n(k) \xrightarrow{n\to\infty} 1 \) in \Cref{thm:Asymptotic} is sharp \emph{uniformly on \( 3 \leq k \leq n! - 3 \)}. However, for fixed values of \( k \), the convergence may be faster. In fact, our proof shows that for every fixed \( k \), there exists a constant \( C_k > 0 \), independent of \( n \), such that
\[
R(n,k) \leq \frac{C_k}{n^{\max\{2, k-4\}}}.
\]
This motivates the following question.

\medskip
\noindent\textbf{Problem 3 (Rate of Convergence).} For a fixed value of \( k \), what is the optimal convergence rate of \( R(n,k) \to 0 \) as \( n \to \infty \)?
\medskip

In the same spirit as the above problem, \Cref{thm:Asymptotic} and its proof also suggest that, for fixed \( n \), the ratio \( T(n,k)/\Lambda_n(k) \) may admit an asymptotic expansion in powers of \( 1/n \) of the form
\[
T(n, k) = \Lambda_n(k) \bigg(1 + \sum_{j \geq 1} A^{(n,k)}_j \frac{1}{n^j} \bigg).
\]
Such expansions are common in the context of partition functions of Gibbs measures in statistical physics (see, e.g.,~\cite{Geloun2022}), and their coefficients often admit combinatorial interpretations. Establishing such an expansion for \( T(n, k) \) would be an interesting direction for further research.

\medskip
\noindent\textbf{Problem 4 (Asymptotic Expansion).} Does the ratio \( T(n, k)/\Lambda_n(k) \) admit an asymptotic expansion in powers of \( 1/n \) for each fixed \( n \)? If so, can the coefficients \( A^{(n,k)}_j \) be given a combinatorial interpretation?
\medskip

Although the bounds in \Cref{lem - partition} and \Cref{lem - ratio} suffice for our purposes, they are not optimal. It is natural to ask whether sharp bounds, or even exact formulas, can be given for \( z_{\nu^m} \) in terms of the cycle counts \( m_k \) of a partition \( \nu \), for arbitrary \( m > 1 \). While such formulas exist, they tend to be complex and difficult to simplify. In special cases (e.g., \( m = 2, 3 \)), we can show that the maximum of \( z_{\nu^m}/z_\nu \) is achieved when \( \nu \) consists of \( t_{n,m} \) cycles of length \( m \), with
\[
\max_{\nu \vdash n} \frac{z_{\nu^m}}{z_\nu} = \frac{n!}{(n - t_{n,m}m)! \, m^{t_{n,m}} \, t_{n,m}!}, ~\quad 
t_{n,m} = \left\lceil \frac{n + 2 - \sqrt[m]{n + 2}}{m} \right\rceil - 1.
\]
Hence, we propose the following problem.

\medskip
\noindent\textbf{Problem 5 (Bounds for \( \bm{z_{\nu^m}} \)).} Can one obtain sharp bounds or exact formulas for \( z_{\nu^m} \) in terms of the cycle counts \( m_k \) of a partition \( \nu \), for arbitrary \( m \geq 1 \)?
\medskip

The transformation \( \nu \mapsto \nu^m \) also induces a map on Young diagrams, via the correspondence between partitions and diagram shapes. The shape of the Young diagram corresponding to random partitions (e.g., chosen according to Plancherel measure) has yielded extremely rich and deep mathematical results. This motivates our last proposed problem.

\medskip
\noindent\textbf{Problem 6 (Random Young Diagrams).} If \( \nu \) is chosen at random (e.g., under Plancherel measure), how does the shape of \( \nu^m \) behave as \( n \to \infty \)? More broadly, can one describe the joint distribution of the sequence of Young diagrams \( \text{shape}(\nu),\, \text{shape}(\nu^2),\, \text{shape}(\nu^3), \ldots \)?

\bibliography{main.bib}

\end{document}